\documentclass{amsart}
\usepackage{amsmath, amssymb}

\title{Warning's Second Theorem with Restricted Variables}
\author{Pete L. Clark}
\author{Aden Forrow}
\author{John R. Schmitt}
\begin{document}
\newtheorem{lemma}{Lemma}[section]
\newtheorem{prop}[lemma]{Proposition}
\newtheorem{cor}[lemma]{Corollary}
\newtheorem{thm}[lemma]{Theorem}
\newtheorem{ques}{Question}
\newtheorem{quest}[lemma]{Question}
\newtheorem{conj}[lemma]{Conjecture}
\newtheorem{fact}[lemma]{Fact}
\newtheorem*{mainthm}{Main Theorem}
\newtheorem{obs}[lemma]{Observation}
\newtheorem{hint}{Hint}
\newtheorem{remark}[lemma]{Remark}
\newtheorem{example}[lemma]{Example}

\newcommand{\pp}{\mathfrak{p}}
\newcommand{\DD}{\mathcal{D}}
\newcommand{\F}{\ensuremath{\mathbb F}}
\newcommand{\Fp}{\ensuremath{\F_p}}
\newcommand{\Fl}{\ensuremath{\F_l}}
\newcommand{\Fpbar}{\overline{\Fp}}
\newcommand{\Fq}{\ensuremath{\F_q}}
\newcommand{\PP}{\mathbb{P}}
\newcommand{\PPone}{\mathfrak{p}_1}
\newcommand{\PPtwo}{\mathfrak{p}_2}
\newcommand{\PPonebar}{\overline{\PPone}}
\newcommand{\N}{\ensuremath{\mathbb N}}
\newcommand{\Q}{\ensuremath{\mathbb Q}}
\newcommand{\Qbar}{\overline{\Q}}
\newcommand{\R}{\ensuremath{\mathbb R}}
\newcommand{\Z}{\ensuremath{\mathbb Z}}
\newcommand{\SSS}{\ensuremath{\mathcal{S}}}
\newcommand{\Rn}{\ensuremath{\mathbb R^n}}
\newcommand{\Ri}{\ensuremath{\R^\infty}}
\newcommand{\C}{\ensuremath{\mathbb C}}
\newcommand{\Cn}{\ensuremath{\mathbb C^n}}
\newcommand{\Ci}{\ensuremath{\C^\infty}}\newcommand{\U}{\ensuremath{\mathcal U}}
\newcommand{\gn}{\ensuremath{\gamma^n}}
\newcommand{\ra}{\ensuremath{\rightarrow}}
\newcommand{\fhat}{\ensuremath{\hat{f}}}
\newcommand{\ghat}{\ensuremath{\hat{g}}}
\newcommand{\hhat}{\ensuremath{\hat{h}}}
\newcommand{\covui}{\ensuremath{\{U_i\}}}
\newcommand{\covvi}{\ensuremath{\{V_i\}}}
\newcommand{\covwi}{\ensuremath{\{W_i\}}}
\newcommand{\Gt}{\ensuremath{\tilde{G}}}
\newcommand{\gt}{\ensuremath{\tilde{\gamma}}}
\newcommand{\Gtn}{\ensuremath{\tilde{G_n}}}
\newcommand{\gtn}{\ensuremath{\tilde{\gamma_n}}}
\newcommand{\gnt}{\ensuremath{\gtn}}
\newcommand{\Gnt}{\ensuremath{\Gtn}}
\newcommand{\Cpi}{\ensuremath{\C P^\infty}}
\newcommand{\Cpn}{\ensuremath{\C P^n}}
\newcommand{\lla}{\ensuremath{\longleftarrow}}
\newcommand{\lra}{\ensuremath{\longrightarrow}}
\newcommand{\Rno}{\ensuremath{\Rn_0}}
\newcommand{\dlra}{\ensuremath{\stackrel{\delta}{\lra}}}
\newcommand{\pii}{\ensuremath{\pi^{-1}}}
\newcommand{\la}{\ensuremath{\leftarrow}}
\newcommand{\gonem}{\ensuremath{\gamma_1^m}}
\newcommand{\gtwon}{\ensuremath{\gamma_2^n}}
\newcommand{\omegabar}{\ensuremath{\overline{\omega}}}
\newcommand{\dlim}{\underset{\lra}{\lim}}
\newcommand{\ilim}{\operatorname{\underset{\lla}{\lim}}}
\newcommand{\Hom}{\operatorname{Hom}}
\newcommand{\Ext}{\operatorname{Ext}}
\newcommand{\Part}{\operatorname{Part}}
\newcommand{\Ker}{\operatorname{Ker}}
\newcommand{\im}{\operatorname{im}}
\newcommand{\ord}{\operatorname{ord}}
\newcommand{\unr}{\operatorname{unr}}
\newcommand{\B}{\ensuremath{\mathcal B}}
\newcommand{\Ocr}{\ensuremath{\Omega_*}}
\newcommand{\Rcr}{\ensuremath{\Ocr \otimes \Q}}
\newcommand{\Cptwok}{\ensuremath{\C P^{2k}}}
\newcommand{\CC}{\ensuremath{\mathcal C}}
\newcommand{\gtkp}{\ensuremath{\tilde{\gamma^k_p}}}
\newcommand{\gtkn}{\ensuremath{\tilde{\gamma^k_m}}}
\newcommand{\QQ}{\ensuremath{\mathcal Q}}
\newcommand{\I}{\ensuremath{\mathcal I}}
\newcommand{\sbar}{\ensuremath{\overline{s}}}
\newcommand{\Kn}{\ensuremath{\overline{K_n}^\times}}
\newcommand{\tame}{\operatorname{tame}}
\newcommand{\Qpt}{\ensuremath{\Q_p^{\tame}}}
\newcommand{\Qpu}{\ensuremath{\Q_p^{\unr}}}
\newcommand{\scrT}{\ensuremath{\mathfrak{T}}}
\newcommand{\That}{\ensuremath{\hat{\mathfrak{T}}}}
\newcommand{\Gal}{\operatorname{Gal}}
\newcommand{\Aut}{\operatorname{Aut}}
\newcommand{\tors}{\operatorname{tors}}
\newcommand{\Zhat}{\hat{\Z}}
\newcommand{\linf}{\ensuremath{l_\infty}}
\newcommand{\Lie}{\operatorname{Lie}}
\newcommand{\GL}{\operatorname{GL}}
\newcommand{\End}{\operatorname{End}}
\newcommand{\aone}{\ensuremath{(a_1,\ldots,a_k)}}
\newcommand{\raone}{\ensuremath{r(a_1,\ldots,a_k,N)}}
\newcommand{\rtwoplus}{\ensuremath{\R^{2  +}}}
\newcommand{\rkplus}{\ensuremath{\R^{k +}}}
\newcommand{\length}{\operatorname{length}}
\newcommand{\Vol}{\operatorname{Vol}}
\newcommand{\cross}{\operatorname{cross}}
\newcommand{\GoN}{\Gamma_0(N)}
\newcommand{\GeN}{\Gamma_1(N)}
\newcommand{\GAG}{\Gamma \alpha \Gamma}
\newcommand{\GBG}{\Gamma \beta \Gamma}
\newcommand{\HGD}{H(\Gamma,\Delta)}
\newcommand{\Ga}{\mathbb{G}_a}
\newcommand{\Div}{\operatorname{Div}}
\newcommand{\Divo}{\Div_0}
\newcommand{\Hstar}{\cal{H}^*}
\newcommand{\txon}{\tilde{X}_0(N)}
\newcommand{\sep}{\operatorname{sep}}
\newcommand{\notp}{\not{p}}
\newcommand{\Aonek}{\mathbb{A}^1/k}
\newcommand{\Wa}{W_a/\mathbb{F}_p}
\newcommand{\Spec}{\operatorname{Spec}}

\newcommand{\abcd}{\left[ \begin{array}{cc}
a & b \\
c & d
\end{array} \right]}

\newcommand{\abod}{\left[ \begin{array}{cc}
a & b \\
0 & d
\end{array} \right]}

\newcommand{\unipmatrix}{\left[ \begin{array}{cc}
1 & b \\
0 & 1
\end{array} \right]}

\newcommand{\matrixeoop}{\left[ \begin{array}{cc}
1 & 0 \\
0 & p
\end{array} \right]}

\newcommand{\w}{\omega}
\newcommand{\Qpi}{\ensuremath{\Q(\pi)}}
\newcommand{\Qpin}{\Q(\pi^n)}
\newcommand{\pibar}{\overline{\pi}}
\newcommand{\pbar}{\overline{p}}
\newcommand{\lcm}{\operatorname{lcm}}
\newcommand{\trace}{\operatorname{trace}}
\newcommand{\OKv}{\mathcal{O}_{K_v}}
\newcommand{\Abarv}{\tilde{A}_v}
\newcommand{\kbar}{\overline{k}}
\newcommand{\Kbar}{\overline{K}}
\newcommand{\pl}{\rho_l}
\newcommand{\plt}{\tilde{\pl}}
\newcommand{\plo}{\pl^0}
\newcommand{\Du}{\underline{D}}
\newcommand{\A}{\mathbb{A}}
\newcommand{\D}{\underline{D}}
\newcommand{\op}{\operatorname{op}}
\newcommand{\Glt}{\tilde{G_l}}
\newcommand{\gl}{\mathfrak{g}_l}
\newcommand{\gltwo}{\mathfrak{gl}_2}
\newcommand{\sltwo}{\mathfrak{sl}_2}
\newcommand{\h}{\mathfrak{h}}
\newcommand{\tA}{\tilde{A}}
\newcommand{\sss}{\operatorname{ss}}
\newcommand{\X}{\Chi}
\newcommand{\ecyc}{\epsilon_{\operatorname{cyc}}}
\newcommand{\hatAl}{\hat{A}[l]}
\newcommand{\sA}{\mathcal{A}}
\newcommand{\sAt}{\overline{\sA}}
\newcommand{\OO}{\mathcal{O}}
\newcommand{\OOB}{\OO_B}
\newcommand{\Flbar}{\overline{\F_l}}
\newcommand{\Vbt}{\widetilde{V_B}}
\newcommand{\XX}{\mathcal{X}}
\newcommand{\GbN}{\Gamma_\bullet(N)}
\newcommand{\Gm}{\mathbb{G}_m}
\newcommand{\Pic}{\operatorname{Pic}}
\newcommand{\FPic}{\textbf{Pic}}
\newcommand{\solv}{\operatorname{solv}}
\newcommand{\Hplus}{\mathcal{H}^+}
\newcommand{\Hminus}{\mathcal{H}^-}
\newcommand{\HH}{\mathcal{H}}
\newcommand{\Alb}{\operatorname{Alb}}
\newcommand{\FAlb}{\mathbf{Alb}}
\newcommand{\gk}{\mathfrak{g}_k}
\newcommand{\car}{\operatorname{char}}
\newcommand{\Br}{\operatorname{Br}}
\newcommand{\gK}{\mathfrak{g}_K}
\newcommand{\coker}{\operatorname{coker}}
\newcommand{\red}{\operatorname{red}}
\newcommand{\CAY}{\operatorname{Cay}}
\newcommand{\ns}{\operatorname{ns}}
\newcommand{\xx}{\mathbf{x}}
\newcommand{\yy}{\mathbf{y}}
\newcommand{\E}{\mathbb{E}}
\newcommand{\rad}{\operatorname{rad}}
\newcommand{\Top}{\operatorname{Top}}
\newcommand{\Map}{\operatorname{Map}}
\newcommand{\Li}{\operatorname{Li}}
\renewcommand{\Map}{\operatorname{Map}}
\newcommand{\ZZ}{\mathcal{Z}}
\newcommand{\uu}{\mathfrak{u}}
\newcommand{\mm}{\mathfrak{m}}
\newcommand{\zz}{\mathbf{z}}

\begin{abstract}
We present a restricted variable generalization of Warning's Second Theorem (a result giving a lower bound on the number of solutions of a low degree polynomial system over a finite field, assuming one solution exists).  This is analogous to Brink's restricted variable generalization of Chevalley's Theorem (a result giving conditions for a low degree polynomial system \emph{not} to have exactly one solution).  Just as Warning's Second Theorem implies Chevalley's Theorem, our result implies Brink's Theorem.  We include several combinatorial applications, enough to show that we have a general tool for obtaining quantitative refinements of combinatorial existence theorems.
\end{abstract}

\maketitle
\noindent
Let $q = p^{\ell}$ be a power of a prime number $p$, and let $\F_q$ be ``the'' finite field of order $q$.  
\\ \\
For $a_1,\ldots,a_n,N \in \Z^+$, we denote by $\mm(a_1,\ldots,a_n;N) \in \Z^+$ a certain combinatorial quantity defined and computed in $\S$ 2.1.

\section{Introduction}
\noindent
A \textbf{C$_1$-field} is a field $F$ such that for all positive integers 
$d < n$ and every homogeneous polynomial $f(t_1,\ldots,t_n) \in F[t_1,\ldots,t_n]$ of degree $d$, there is $x \in F^n \setminus \{(0,\ldots,0)\}$ such that $f(x) = 0$.  This notion is due to E. Artin.  However, already in 1909 L.E. Dickson had conjectured that (in Artin's language) 
every finite field is a $C_1$-field \cite{Dickson09}.  Tsen showed that function fields in one variable over an algebraically closed field are $C_1$-fields \cite{Tsen33}, but this left the finite field case open.  Artin assigned the problem of proving Dickson's conjecture to his student Ewald 
Warning.  In 1934 C. Chevalley visited Artin, asked about his student's work, and quickly proved a result which implies that finite fields are $C_1$-fields.  In danger of losing his thesis problem, Warning responded by establishing a further improvement.  The papers of Chevalley and Warning were published consecutively \cite{Chevalley35}, \cite{Warning35}, and the following result is now a classic of elementary number theory.

\begin{thm}(Chevalley-Warning Theorem)
\label{WARTHM}
Let $n,r,d_1,\ldots,d_r \in \Z^+$ with 
\begin{equation}
\label{SMALLDEGREE}
d_1 + \ldots + d_r < n. 
\end{equation}
  For $1 \leq i \leq r$, let 
$P_i(t_1,\ldots,t_n) \in \F_q[t_1,\ldots,t_n]$ be a polynomial of degree $d_i$.  Let 
\[ Z = Z(P_1,\ldots,P_r) = \{x \in \F_q^n \mid P_1(x) = \ldots = P_r(x) = 0\} \]
be the common zero set in $\F_q^n$ of the $P_i$'s, and let $\zz = \# Z$.  Then: \\
a) (Chevalley's Theorem \cite{Chevalley35}) We have $\zz = 0$ or $\zz \geq 2$.  \\
b) (Warning's Theorem \cite{Warning35}) We have $\zz \equiv 0 \pmod{p}$.
\end{thm}

\noindent
In fact very easy modifications of Chevalley's argument prove Warning's Theorem.  
 The more substantial contribution of \cite{Warning35} is the following result.

\begin{thm}(Warning's Second Theorem)
\label{WARNING2}
With hypotheses as in Theorem \ref{WARTHM}, 
\begin{equation}
\label{WAR2EQ}
 \zz = 0 \text{ or } \zz \geq q^{n-d}. 
\end{equation}
\end{thm}

\noindent
There is a rich body of work on extensions and refinements of Theorem \ref{WARTHM} -- too much to recall here! -- but let us mention work of Ax and Katz which computes 
the minimal $p$-adic valuation of $\zz$ as $P_1,\ldots,P_r$ range over all polynomials of degrees $d_1,\ldots,d_r$ and work of Esnault showing that various geometric classes of varieties -- including all Fano varieties -- over finite fields must have rational points \cite{Ax64}, \cite{Katz71},  \cite{Esnault03}.  In contrast we know of only one refinement of Theorem \ref{WARNING2}: \cite{Heath-Brown11}. 
\\ \\
The above generalizations of the Chevalley-Warning Theorem point in the direction of arithmetic geometry.  Here we are more interested in interfaces with combinatorics.  Here is the first result in this direction.

\begin{thm}(Schanuel's Theorem \cite{Schanuel74})
\label{SCHANTHM}
Let $n,r,v_1,\ldots,v_r \in \Z^+$.  For $1 \leq j \leq r$, let $P_j(t_1,\ldots,t_n) \in \Z/p^{v_j} \Z[t_1,\ldots,t_n]$ be a polynomial  without constant term. Let
\[ Z^{\circ} = \{ x \in \Z^n \setminus (p\Z)^n \mid P_j(x) \equiv 0 \pmod{p_j^{v_j}} \text{ for all } 1 \leq j \leq r \}. \]
a) If $\sum_{j=1}^r\deg(P_j) \left(\frac{p^{v_j}-1}{p-1}\right) < n$, then $Z^{\circ} \neq \varnothing$.  \\
b) If $\sum_{j=1}^r (p^{v_j}-1)\deg(P_j)  < n$, then  $Z^{\circ} \cap \{0,1\}^n \neq \varnothing$.  \\
c) Let $b_1,\ldots,b_n$ be non-negative integers.  If $\sum_{j=1}^r 
(p^{v_j}-1)\deg(P_j) < \sum_{i=1}^n b_i$, then $Z^{\circ} \cap \prod_{i=1}^n [0,b_i] \neq \varnothing$.  
\end{thm} 
\noindent  
These results have been revisited in light of the \textbf{Polynomial Method}, a technique initiated by N. Alon \cite{Alon99} and continued by many others.  The first application in \cite{Alon99} is to Chevalley's Theorem.
Recently U. Schauz \cite{Schauz08a} and then D. Brink \cite{Brink11} used Alon's ideas to prove a \textbf{restricted variable} generalization.

\begin{thm}(Restricted Variable Chevalley Theorem) \label{RESVARCHEV}
Let $P_1,\ldots,P_r \in \F_q[t] = \F_q[t_1,\ldots,t_n]$.  For $1 \leq i \leq n$, let $\varnothing \neq A_i \subseteq \F_q$ and put $A = \prod_{i=1}^n A_i$. 
Put \[Z_A = \{ a = (a_1,\ldots,a_n) \in A \mid 
P_1(a) = \ldots = P_r(a) = 0 \}, \ \zz_A = \# Z_A. \]
If $\sum_{i=1}^r (q-1)\deg P_i < \sum_{i=1}^n(\# A_i - 1)$,
then $\zz_A \neq 1$. 
\end{thm}
\noindent
Brink gave a common generalization of 
Theorem \ref{SCHANTHM} and of Theorem \ref{RESVARCHEV} for $q = p$.

\begin{thm}(Brink's Theorem \cite{Brink11})
\label{RVSCHANUEL}
\label{BIGBRINKTHM}
Let $P_1(t_1,\ldots,t_n),\ldots,P_r(t_1,\ldots,t_n) \in \Z[t_1,\ldots,t_n]$ be 
polynomials, let $p$ be a prime, let $v_1,\ldots,v_r \in \Z^+$, and 
let $A_1,\ldots,A_n$ be nonempty subsets of $\Z$ such that for each $i$, 
the elements of $A_i$ are pairwise incongruent modulo $p$, and put 
$A = \prod_{i=1}^n A_i$. 
 Let 
\[ Z_A = \{x \in A \mid P_j(x) \equiv 0 \pmod{p^{v_j}} \ \forall 1 \leq j \leq r \}, \ \zz_A = \# Z_A. \]
a) If $\sum_{j=1}^r (p^{v_j}-1)\deg(P_j) < \sum_{i=1}^n \left( \#A_i - 1 \right)$, then $\zz_A \neq 1$. \\
b) (\textbf{Boolean Case}) If $A = \{0,1\}^n$ and $\sum_{j=1}^r (p^{v_j}-1)\deg(P_j) < n$, then $\zz_A \neq 1$. 
\end{thm}
\noindent
Following a remark of Brink, we state in $\S$ 3.3 a generalization to number fields, Theorem \ref{NFBRINKTHM}, which fully recovers Theorem \ref{RESVARCHEV}.
\\ \\
The main result of this paper simultaneously generalizes Theorems  \ref{WARNING2} and \ref{NFBRINKTHM}.

\begin{thm}(Restricted Variable Warning's Second Theorem)
\label{SRVW2THM}
Let $K$ be a number field with ring of integers $R$, let $\pp$ be a nonzero prime ideal of $R$, and let $q = p^{\ell}$ be the prime power such that 
$R/\pp \cong \F_q$.  Let $A_1,\ldots,A_n$ be nonempty subsets of $R$ such that for each $i$, 
the elements of $A_i$ are pairwise incongruent modulo $\pp$, and put $A = \prod_{i=1}^n A_i$.  Let $r,v_1,\ldots,v_r \in \Z^+$.  Let $P_1,\ldots,P_r \in R[t_1,\ldots,t_n]$.  Let 
\[ Z_A = \{ x \in A \mid P_j(x) \equiv 0 \pmod{\pp^{v_j}} \ \forall 1 \leq j \leq r\}, \ \zz_A = \# Z_A. \]
a) $\zz_A = 0$ or $\zz_A \geq \mm \left( \# A_1,\ldots, \# A_n; \# A_1 + \ldots + \# A_n - 
\sum_{j=1}^r (q^{v_j} -1) \deg(P_j) \right)$. \\
b) We recover Theorem \ref{WARNING2} and Theorem \ref{NFBRINKTHM} as special cases. \\
c) (\textbf{Boolean Case}) We have  $\zz_{\{0,1\}^n} = 0$ or $\zz_{\{0,1\}^n} \geq 2^{n-\sum_{j=1}^r (q^{v_j} -1) \deg(P_j)}$.
\end{thm}
\noindent
Theorem \ref{SRVW2THM} includes all of the results stated so far except Theorem \ref{WARTHM}b). In this regard we should first mention that J. Ax gave a 
\emph{ten line proof} of Theorem \ref{WARTHM}b) \cite{Ax64}.  Chevalley's original proof is longer but seems more penetrating: 
it adapts easily to give a restricted variable generalization of Theorem \ref{WARTHM}b):  see \cite[Thm. 16]{Clark14}.  Adapting Chevalley's method for finitely restricted variables over an arbitrary field leads to a \textbf{Coefficient Formula} which has appeared in the recent literature 
\cite[Thm. 3.2]{Schauz08a}, \cite[Thm. 3]{Lason10}, \cite[Thm. 4]{Karasev-Petrov12}, \cite[$\S$ 3.3]{Clark14} as a natural sharpening of Alon's Combinatorial Nullstellensatz II 
\cite[Thm. 1.2]{Alon99}.  
\\ \indent
By whatever name, the above Polynomial Method is the key to the proof of the above results of Schanuel, Schauz and Brink.  The key to the proof of the Restricted Variable Warning's Second Theorem is a \emph{different} Polynomial Method: the \textbf{Alon-F\"uredi Theorem}.  $\S$ 2 of this paper recalls the statement of this theorem and gives some other needed preliminaries of both a combinatorial 
and number-theoretic nature.  The proof of Theorem \ref{SRVW2THM} occurs in 
$\S$ 3 (the shortest section!).  
\\ \\
Chevalley's Theorem has some combinatorial applications, notably the Theorem of 
Erd\H os, Ginzburg and Ziv (henceforth EGZ).  Schanuel's refinement has a very striking application in additive combinatorics: it yields a theorem of Olson computing the Davenport constant of a finite commutative $p$-group.  Further, it is the main technical input of a result of Alon, Kleitman, Lipton, Meshulam, Rabin and Spencer (henceforth AKLMRS) on selecting from set systems to get a union of cardinality divisible by a prime power $q$.  As Brink shows, his Theorem \ref{BIGBRINKTHM} can be applied in additive combinatorics to convert theorems asserting the existence of subsequences into theorems asserting the existence of ``generalized subsequences'' formed by taking linear combinations with coefficients in a restricted variable set.  This is a natural generalization, going back at least as far as the Shannon capacity: c.f. \cite{Mead-Narkiewicz82}.  Analogues of the EGZ Theorem in the context of generalized subsequences (or ``weighted subsequences'') in $p$-groups are pursued in the recent work \cite{DAGS12} of Das Adhikari, Grynkiewicz and Sun (henceforth DAGS).
\\ \\
In $\S$ 4 we apply Theorem \ref{SRVW2THM} to each of the above situations, getting in each case a quantitative refinement which also includes the \textbf{inhomogeneous case}: thus whereas Brink gave an upper bound on the length of a sequence in a $p$-group $G$ with no generalized $0$-sum subsequence, we give a lower bound on the \emph{number} of $g$-sum generalized subsequences (for any $g \in G$) which recovers Brink's result when we specialize to $g = 0$ and ask only for one nontrivial subsequence.  On the other hand, specializing to the case of ``classical'' $g$-sum subsequences we recover a recent result of Chang, Chen, Qu, Wang and Zhang (henceforth CCQWZ) which was proved via combinatorial means \cite{CCQWZ11}.  We give similar refinements of the results of AKLMRS and Das Adhikari, Grynkiewiz and Sun.  
\\ \\
We hope these combinatorial results will be of interest.  But more than any single application, our main goal is to demonstrate that Theorem \ref{SRVW2THM} is a tool that can be broadly applied to refine combinatorial existence theorems into theorems which give explicit (and sometimes sharp) lower bounds on the \emph{number} of combinatorial objects asserted to exist and to treat inhomogeneous cases with results in which the lower bounds are conditional on the existence of any objects of a given type (a plainly necessary restriction in many natural situations).  We tried to find applications which are substantial enough to serve as a true ``proof of concept,'' and we hope to convince the reader that this tool can be a useful one for researchers in branches of mathematics where polynomial methods are currently being applied.

\section{Preliminaries}

\subsection{Balls in Bins}
\textbf{} \\ \\ \noindent
Let $n \in \Z^+$, and let $a_1 \geq \ldots \geq a_n \geq 1$ be integers.  Consider bins $A_1,\ldots,A_n$ such that $A_i$ can hold at most $a_i$ balls.  
For $N \in \Z^+$, a \textbf{distribution of N balls in the bins} $A_1,\ldots,A_n$ is an $n$-tuple $y = (y_1,\ldots,y_n)$ with $y_1 + \ldots + y_n = N$ and $1 \leq y_i \leq a_i$ for all $i$.  Such distributions exist if and only if $n \leq N \leq a_1 + \ldots + a_n$. 
\\ \indent For a distribution $y$ of $N$ balls into bins $A_1,\ldots,A_n$, let
$P(y) = y_1 \cdots y_n$.  If $n \leq N \leq a_1 + \ldots + a_n$, let $\mathfrak{m}(a_1,\ldots,a_n;N)$ be the minimum value of $P(y)$ as $y$ ranges over all distributions of $N$ balls into bins $A_1,\ldots,A_n$.   We have
$\mathfrak{m}(a_1,\ldots,a_n;n) = 1$.  We extend constantly to the left: if $N \in \Z$ is such that $N < n$, put $\mathfrak{m}(a_1,\ldots,a_n;N) = 1$.  Similarly, 
we have $\mathfrak{m}(a_1,\ldots,a_n;a_1+\ldots+a_n) = a_1 \cdots a_n$.  We extend 
constantly to the right: if $N \in \Z$ is such that $N > a_1+\ldots+a_n$, put 
$\mathfrak{m}(a_1,\ldots,a_n;N)= a_1 \cdots a_n$.  Note that if $N_1 \leq N_2$ 
then $\mathfrak{m}(a_1,\ldots,a_n;N_1) \leq \mathfrak{m}(a_1,\ldots,a_n;N_2)$.  

\begin{lemma}
\label{OBVIOUSLEMMA}
Let $n,a_1,\ldots,a_n \in \Z^+$ with $\max \{a_1,\ldots,a_n\} \geq 2$.  Let $N > n$ be an integer.  Then $\mathfrak{m}(a_1,\ldots,a_n;N) \geq 2$.  
\end{lemma}
\begin{proof} This is, literally, the pigeonhole principle.
\end{proof}
\noindent
The following simple result describes the minimal distribution in all cases and thus essentially computes $\mathfrak{m}(a_1,\ldots,a_n;N)$.  A formula in the general case would be unwieldy, but we give exact formulas in some special cases that we will need later.
 
\begin{lemma}
\label{FSLEMMA}
Let $n \in \Z^+$, and let $a_1 \geq \ldots \geq a_n \geq 1$ be integers.  
Let $N$ be an integer with $n \leq N \leq a_1 + \ldots + a_n$.  \\
a) We define the \textbf{greedy configuration} $y_G = (y_1,\ldots,y_n)$: after placing one ball in each bin, place the remaining balls into bins from left to right, filling each bin completely before moving on to the next bin, until we run out of balls.  Then 
\[ \mathfrak{m}(a_1,\ldots,a_n;N) = P(y_G) = y_1 \cdots y_n. \]
b) Suppose $a_1 = \ldots = a_n = a \geq 2$.  If $n \leq N \leq an$, then
\[ \mm(a,\ldots,a;N) = (R+1)a^{\lfloor \frac{N-n}{a-1} \rfloor}, \]
where $R \equiv N-n \pmod{a-1}$ and $0 \leq R < a-1$.  \\
c) For all non-negative integers $k$, we have 
\[\mathfrak{m}(2,\ldots,2;2n-k) = 2^{n-k}. \]
\end{lemma}
\begin{proof}
a) Consider the following two kinds of ``elementary moves'' which transform 
one distribution $y$ of $N$ balls in bins of size $a_1 \geq \ldots \geq a_n \geq 1$ into another $y'$: \\
(i) (Bin Swap): If for $i < j$ we have $y_i < y_j$, then let $y'$ be obtained 
from $y$ by interchanging the $i$th and $j$th coordinates.  Then $P(y') = P(y)$.  \\
(ii) (Unbalancing Move): Suppose that for $1 \leq i \neq j \leq n$ we have 
$1 <  y_i \leq y_j < a_j$.  Then we may remove a ball form the $i$th bin and 
place it in the $j$th bin to get a new distribution $y' = (y_1',\ldots,y_n')$ and 
\[ P(y') = \frac{y_i' y_j'}{y_i y_j} P(y) = \frac{y_iy_j + y_i-y_j -1}{y_i y_j} P(y)< P(y). \]
Starting with any distribution $y$, we may perform a sequence of bin swaps 
to get a distribution $y'$ with and $y_1' \geq \ldots \geq y_n'$ 
and then a sequence of unbalancing moves, each of which has $i$ maximal 
such that $1 < y_i$ and $j$ minimal such that $y_j < a_j$, to arrive at the greedy configuration 
$y_G$.  Thus $P(y) = P(y') \geq P(y_G)$.  \\
b)   Put $k = \lfloor \frac{N-n}{a-1} \rfloor$, so via division with remainder we have
\[ N-n =  k(a-1) + R. \]
The greedy configuration is then 
\[ y_G = (\overbrace{a,\ldots,a}^{k },R+1,\overbrace{1,\ldots,1}^{n-k-1} ). \]
c) This is the special case $a = 2$ of part b). 
\end{proof}

\subsection{The Alon-F\"uredi Theorem}

\begin{thm}(Alon-F\"uredi Theorem) Let $F$ be a field, let $A_1,\ldots,A_n$ 
be nonempty finite subsets of $F$.  Put $A = \prod_{i=1}^n A_i$ and $a_i = \# A_i$ for all $1 \leq i \leq n$.  Let $P \in F[t] = F[t_1,\ldots,t_n]$ be a polynomial.  Let 
\[\mathcal{U}_A = \{x \in A \mid P(x) \neq 0\}, \ \mathfrak{u}_A = \# \mathcal{U}_A. \]
Then  $\uu_A = 0$ or $\uu_A \geq \mathfrak{m}(a_1,\ldots,a_n;a_1+\ldots+a_n - \deg P)$.
\end{thm}
\begin{proof}
See \cite[Thm. 5]{Alon-Furedi93}.
\end{proof}

\subsection{The Schanuel-Brink Operator}
\textbf{} \\ \\ \noindent
Let $p$ be a prime number. For $1 \leq i \leq n$, let $A_i$ be a set of coset representatives of $p\Z$ in $\Z$; put $A = \prod_{i=1}^n A_i$.  In \cite{Schanuel74}, Schanuel proves the following result.


\begin{lemma}
\label{SCHANLEMMA}
Let $v \in \Z^+$, and let $f \in \Z/p^v \Z[t] = \Z/p^v \Z[t_1,\ldots,t_n]$ be a polynomial of degree $d$.  There are polynomials  
$f_1,\ldots,f_v \in \Z/p\Z[t]$ of degrees $d,pd,\ldots,p^{v-1} d$ such that 
for all $x \in A$, $f(x) \equiv 0 \pmod{p^v}$ iff $f_i(x) \equiv 0 \pmod{p}$ for all $1 \leq i \leq v$.
\end{lemma}
\noindent
Since the sum of the degrees of the $f_i$'s in Lemma \ref{SCHANLEMMA} is 
$d+pd+ \ldots + p^{v-1} d = (\frac{p^v-1}{p-1})d$, Lemma \ref{SCHANLEMMA} reduces Theorem \ref{SCHANTHM}a) to the $q = p$ case of Chevalley's Theorem. 
\\ \\
Although the statement concerns only finite rings, all known proofs use characteristic $0$ constructions.  Schanuel's proof works in the ring of $p$-adic integers 
$\Z_p = \varprojlim \Z/p^n \Z$: as he mentions, it is really motivated by the theory of \textbf{Witt vectors} but can be -- and was -- presented in a self-contained way.  In \cite{Brink11}, Brink generalized and simplified Schanuel's construction (actually some of Brink's simplifications have already been incorporated in our statement of Lemma \ref{SCHANLEMMA}; Schanuel spoke of solutions with coordinates in the set of Teichm\"uller representatives for $\F_p$ in $\Z_p$) by working in the localization of $\Z$ at the 
prime ideal $(p)$, namely 
\[ \Z_{(p)} = \left\{ \frac{a}{b} \in \Q \text{ such that } p \nmid b \right\}. \]
Following Schanuel, Brink introduces an operator (which depends on the choice of $A$, though we suppress it from the notation)
\[ \Delta: \Z_{(p)}[t] \ra \Z_{(p)}[t] \]
such that $\deg \Delta(f) \leq p \deg f$ and for all $x \in A$ we have $f(x) \equiv 0 \pmod{p^v}$ iff $(\Delta^i f)(x) \equiv 0 \pmod{p}$ 
for $0 \leq i \leq v-1$.  This is all we need to prove Theorem \ref{SRVW2THM} in the $q = p$ case.  Since this is the only case which gets applied in $\S$ 4, readers who are more interested in combinatorics than algebraic number theory may wish to move on to the next section.  However, we wish to state Theorem \ref{SRVW2THM} so that it includes Warning's Second Theorem over $\F_q$ and to deduce a suitable strengthening of Brink's Theorem from it, and this necessitates the following setup.  
\\ \\
Let $K$ be a number field with ring of integers $R$.  Let $\pp$ be a prime ideal of $R$, so $R/\pp \cong \F_q$ for a prime power $q = p^{\ell}$.  Let $R_{\pp}$ 
be the localization of $R$ at the prime ideal $\pp$, which is a discrete valuation ring with discrete valuation $v_{\pp}$.  Let $\pi$ in $R$ be 
such that $v_{\pp}(\pi) = 1$, so $\pp R_{\pp} = \pi R_{\pp}$.  
\\ \\
For $1 \leq i \leq n$, let $\varnothing \neq A_i \subset R$ be such that 
distinct elements of $A_i$ are incongruent modulo $\pp$.  (So $\# A_i \leq q$ for all $i$.)  Put $A = \prod_{i=1}^n A_i$.  For $1 \leq i \leq n$, there is $\tau_i(x) \in K[x]$ of degree less than $q$ such that $\tau_i(a_i) = \frac{a_i-a_i^q}{\pi}$ for all $a_i \in A_i$: 
\[ \tau_i(x) = \sum_{a_i \in A_i} \frac{a_i-a_i^q}{\pi} \prod_{b_i \in A_i \setminus \{a_i\}} \frac{x-b_i}{a_i-b_i}. \]
This formula makes clear that $\tau_i(x) \in R_{\pp}[x]$.  For $1 \leq i \leq n$, put 
\[ \sigma_i(x) = x^q + \pi \tau_i(x). \]
It follows that: \\
$\bullet$ $\sigma_i(x) \in R_{\pp}[x]$; \\
$\bullet$ $\deg \sigma_i = q$; \\
$\bullet$ for all $a_i \in A_i$, $\sigma_i(a_i) = a_i$; and \\
$\bullet$ $\sigma_i(x) \equiv x^q \pmod {\pp R_{\pp}[x]}$.
\\ \\
We define the \textbf{Schanuel-Brink operator} $\Delta: K[t_1,\ldots,t_n] \ra K[t_1,\ldots,t_n]$ by 
\[ \Delta:  
f(t_1,\ldots,t_n) \mapsto \frac{f(t_1,\ldots,t_n)^q-f(\sigma_1(t_1),\ldots,\sigma_n(t_n))}{\pi}. \]

\begin{lemma}(Properties of the Schanuel-Brink Operator) 
\label{DELTALEMMA} \\
a) For all $f \in K[t]$, $\deg \Delta(f) \leq q \deg f$. \\
b) If $c \in K$, then $\Delta(c) = \frac{c^q-c}{\pi}$. \\
c) For all $f \in R_{\pp}[t]$, we have $\Delta(f) \in R_{\pp}[t]$. \\
d) For all $f \in R_{\pp}[t]$, $a = (a_1,\ldots,a_n) \in A$, $i \in \Z^+$, we have $(\Delta^i f)(a) = \Delta^i (f(a))$.  \\
e) For all $c \in R_{\pp}$ and $v \in \Z^+$, the following are equivalent: \\
(i) $c \equiv 0 \pmod{\pp^v}$.  \\
(ii) We have $c,\Delta c,\ldots, \Delta^{v-1}(c) \equiv 0 \pmod \pp$.  
\end{lemma}
\begin{proof}
Parts a) and b) are immediate.  \\
c) It is enough to show that the image in $\F_q[t]$ of 
$f(t)^q-f(\sigma_1(t_1),\ldots,\sigma_n(t_n))$ is zero.  In characteristic $p$ we have $(x+y)^p = x^p + y^p$, and applying this $a$ times gives $(x+y)^q = x^q + y^q$.  Since also $a^q = a$ for all $a \in \F_q$ it follows that for any 
\[ f(t) = \sum_I c_I t_1^{a_1} \cdots t_n^{a_n} \] 
we have that as elements of $\F_q[t]$,
\[ f(t)^q = \sum_I c_I t_1^{qa_1} \cdots t_n^{qa_n} = f(\sigma_1(t_1),\ldots,\sigma_n(t_n)). \] 
d) Since $\sigma_i(a_i) = a_i$ for all $a_i \in A_i$, 
\[ (\Delta f)(a) = \frac{ f(a_1,\ldots,a_n)^q - f(a_1,\ldots,a_n)}{\pi} = 
\Delta (f(a)),\]
establishing the $i = 1$ case.  The general case follows by induction. \\
e) If $c = 0$ then (i) and (ii) hold.  Each of (i) and (ii) implies $c \equiv 0 \pmod{\pp}$, so we may assume $c \neq 0$ and $c \equiv 0 \pmod{\pp}$.  Since $c \equiv 0 \pmod{\pp}$, 
$v_\pp(c) \geq 1$ and thus \[v_\pp(c^q) = q v_\pp(c) > v_\pp(c). \]  It follows that 
$v_\pp(c^q-c) = v_\pp(c)$ (if $\pp^{v_\pp(c)+1}$ divided $c^q-c$, then it would divide $c^q$ and hence it would divide $c$, contradiction) and thus 
\[ v_{\pp}(\Delta(c)) = v_\pp\left(\frac{c^q-c}{\pi}\right) = v_\pp(c^q-c) - 1 = v_\pp(c) -1.\]
The equivalence (i) $\iff$ (ii) follows.
\end{proof}
\noindent
The following immediate consequence is the main result of this section.

\begin{cor}
\label{DELTACOR} 
For all $f \in R_{\pp}[t]$, $a \in A$ 
and $v \in \Z^+$, we have $f(a) \equiv 0 \pmod{\pp^v}$ iff 
$(\Delta^i f)(a) \equiv 0 \pmod{\pp}$ for all $0 \leq i \leq v-1$. 
\end{cor}


\section{The Restricted Variable Warning's Second Theorem}

\subsection{Brink's Theorem in a Number Field}

\begin{thm}
\label{NFBRINKTHM}
Let $K$ be a number field with ring of integers $R$, let $\pp$ be a nonzero prime ideal of $R$, and let $q = p^{\ell}$ be the prime power such that 
$R/\pp \cong \F_q$.  Let $P_1(t_1,\ldots,t_n),\ldots,P_r(t_1,\ldots,t_n) \in R[t_1,\ldots,t_n]$, let $v_1,\ldots,v_r \in \Z^+$, and 
let $A_1,\ldots,A_n$ be nonempty subsets of $R$ such that for each $i$, 
the elements of $A_i$ are pairwise incongruent modulo $\pp$, and put 
$A = \prod_{i=1}^n A_i$.  Let 
\[ Z_A = \{x \in A \mid P_j(x) \equiv 0 \pmod{\pp^{v_j}} \ \forall 1 \leq j \leq r \}, \ \zz_A = \# Z_A. \]
a) If $\sum_{j=1}^r (q^{v_j}-1)\deg(P_j) < \sum_{i=1}^n \left( \#A_i - 1 \right)$, then $\zz_A \neq 1$. \\
b) (\textbf{Boolean Case}) If $A = \{0,1\}^n$ and 
\[ \sum_{j=1}^r (q^{v_j}-1)\deg(P_j) < n, \]
then $\zz_A \neq 1$.  
\end{thm}  
\noindent
Brink states and proves Theorem \ref{BIGBRINKTHM} in the $K = \Q$ case \cite[Thm. 2]{Brink11}.  Theorem \ref{NFBRINKTHM} is stated on page 130 of his paper.  Having carried over the Schanuel-Brink operator to number fields, Brink's proof applies verbatim.  Rather than replicate this argument, we will deduce Theorem \ref{NFBRINKTHM} as a consequence of Theorem \ref{SRVW2THM}.


\subsection{Proof of The Restricted Variable Warning's Second Theorem}


\begin{proof} 
a) Step 1: Suppose each $v_i = 1$.  Put 
\[ P(t) = \prod_{i=1}^r (1-P_i(t)^{q-1}). \]
Then $\deg P = (q-1)(\deg(P_1) + \ldots + \deg(P_r))$, and 
\[ \mathcal{U}_A = \{x \in A \mid P(x) \neq 0\} = Z_A, \]
so 
\[ z_A = \# Z_A = \# \mathcal{U}_A = \uu_A. \]
Applying the Alon-F\"uredi Theorem we get $\zz_A = 0$ or 
\[ \zz_A \geq \mathfrak{m}(\# A_1+\ldots+\# A_n;\# A_1 + \ldots + \# A_n - (q-1)d). \]
Step 2: Let $a \in A$ and $f \in R_{\pp}[t_1,\ldots,t_n]$.  By Corollary \ref{DELTACOR},
\[ f(a) \equiv 0 \pmod{q^{v_i}}  \iff (\Delta^i f)(a) \equiv 0 \ \forall \ i \leq v_i -1. \]
Moreover, by Lemma \ref{DELTALEMMA}a), $\deg \Delta^i f \leq q^i \deg f$.  Thus 
for each $1 \leq j \leq r$, we have exchanged the congruence $P_j \equiv 0 \pmod{\pp^{v_j}}$ for the system of congruences 
\[ P_j \equiv 0 \pmod{\pp}, \ \Delta P_j \equiv 0 \pmod{\pp}, \ldots, 
\Delta^{v_j-1} P_j \equiv 0 \pmod{\pp} \]
of degrees at most $\deg P_j,\  q \deg P_j,\ldots, q^{v_j-1} \deg P_j$.  Hence
the sum of the degrees of all the polynomial congruences is at most 
\[ \sum_{j=1}^r (1+q+ \ldots + q^{v_j-1}) \deg P_j = \sum_{j=1}^r \frac{q^{v_j}-1}{q-1}\deg(P_j). \]
%
%
%
Apply Step 1.  \\
b) To recover Theorem \ref{WARNING2}: for all $i$, take $A_i$ to be a set of coset representatives for $\pp R$ in $R$, so 
$\# A_i = q$ for all $i$.  Let $k = n- (d_1+\ldots+d_r) = n-d$, so
\[ \# A_1+\ldots+\# A_n - \deg P = nq - (q-1)d = kq+n-k . \]
Lemma \ref{FSLEMMA}b) gives
\[ \mathfrak{m}(\# A_1,\ldots,\# A_n;\# A_1 + \ldots + \# A_n - \deg P) = 
\mathfrak{m}(q,\ldots,q;kq+n-k) 
= q^k = q^{n-d}. \] 
To recover Theorem \ref{NFBRINKTHM}: apply Lemma \ref{OBVIOUSLEMMA} and 
part a).  \\
c) For all $i$ take $A_i = \{0,1\}$.  Lemma \ref{FSLEMMA}c) gives 
\[ \mathfrak{m}(\# A_1,\ldots,\# A_n;\# A_1 + \ldots + \# A_n - \deg P) = 
\mathfrak{m}(2,\ldots,2;2n- \sum_{j=1}^r (q^{v_j}-1)\deg(P_j)) \] \[ =  2^{n-\sum_{j=1}^r (q^{v_j} -1) \deg(P_j)}. \qedhere \] 
\end{proof}

\subsection{Deductions From the Unrestricted Cases}
\textbf{} \\ \\ \noindent
Schanuel proved part b) of Theorem \ref{SCHANTHM} by applying part a) to the polynomials $P_j(t_1^{p-1},\cdots,t_n^{p-1})$: this works since for all 
$x \in \F_p$, $x^{p-1} \in \{0,1\}$.  He proved part c) by applying part a) 
to the polynomials $P_j(t_{1,1}^{p-1} + \ldots + t_{1,b_1}^{p-1},\ldots,
t_{n,1}^{p-1} + \ldots + t_{n,b_n}^{p-1})$ in the $b_1 + \ldots + b_r$ variables 
$t_{1,1},\ldots,t_{1,b_1},\ldots,t_{n,1},\ldots,t_{n,b_n}$.  In particular, the case of Theorem \ref{SCHANTHM}b) in which all congruences are 
modulo $p$ is reduced to Chevalley's Theorem.  This substitution underlies many of the combinatorial applications of the Chevalley-Warning Theorem, e.g. \cite{Bailey-Richter89}: see $\S$ 4.4. 

\begin{ques}
\label{QUES1}
For which $A = \prod_{i=1}^n A_i \subset \F_q^n$ can one deduce the Restricted Variable Chevalley Theorem (Theorem \ref{RESVARCHEV}) from its unrestricted version (Theorem \ref{WARTHM}a))? 
\end{ques}
\noindent
We turn to Warning's Second Theorem.  Since the bound obtained in Theorem \ref{SRVW2THM} is in terms of the combinatorially defined 
quantity $\mm(a_1,\ldots,a_n;N)$, it is natural to wonder to what extent Theorem \ref{SRVW2THM} could be deduced from Theorem \ref{WARNING2} 
by purely combinatorial arguments.  Consider again $A = \{0,1\}^n$.  It turns out that some work has been done on this problem: in Theorem \ref{SRVW2THM}c), take 
$r = v_1 = 1$ and $q = p$, write $P$ for $P_1$, and put $d = \deg P$, so 
\begin{equation}
\label{DEDUCTIONEQ1}
\zz_{ \{0,1\}^n} = 0 \text{ or } \zz_{ \{0,1\}^n} \geq 2^{n-(p-1)d}. 
\end{equation}
Using Warning's Second Theorem and purely combinatorial 
arguments, Chattopadhyay, Goyal, Pudl\'ak and Th\'erien showed \cite[Thm. 11]{CGPT06} that 
\begin{equation}
\label{DEDUCTIONEQ2} \zz_{ \{0,1\}^n} = 0 \text{ or } \zz_{ \{0,1\}^n} \geq 2^{n- (\log_2p)(p-1)d}. \end{equation}
For $p = 2$, (\ref{DEDUCTIONEQ1}) and (\ref{DEDUCTIONEQ2}) coincide with (\ref{WAR2EQ}).  For $p > 2$, (\ref{DEDUCTIONEQ1}) is an improvement of 
(\ref{DEDUCTIONEQ2}).

\section{Combinatorial Applications}

\subsection{The Davenport Constant and $g$-Sum Subsequences}
\textbf{} \\ \\ \noindent
Let $(G,+)$ be a nontrivial finite commutative group.  For $n \in \Z^+$, let 
$x = (x_1,\ldots,x_n) \in G^n$.  We view $x$ as a length $n$ sequence 
$x_1,\ldots,x_n$ of elements in $G$ and a subset $J \subset \{1,\ldots,n\}$ as giving a subsequence $x_J$ of $x$.  For $g \in G$, we say $x_J$ is a \textbf{g-sum subsequence} if $\sum_{i \in J} x_i = g$.  When $g = 0$ we speak of \textbf{zero-sum subsequences}.  
\\ \\
The \textbf{Davenport constant} $D(G)$ is the least $d \in \Z^+$ such that every $x \in G^d$ has a nonempty zero-sum subsequence.  The pigeonhole principle gives
\begin{equation}
\label{DAVENPORTEQ1}
 D(G) \leq \# G. 
\end{equation}
The Davenport constant arises naturally in the theory of factorization in integral domains.  We mention one result to show the flavor.

\begin{thm}
Let $K$ be a number field, let $R$ be its ring of integers, and let 
$\operatorname{Cl} R$ be the ideal class group of $R$.  For $x \in R$ 
nonzero and not a unit, let $L(x)$ (resp. $l(x)$) be the maximum (resp. the minimum) of all lengths of factorizations of $x$ into \emph{irreducible elements}, let 
\[ \rho(x) = \frac{L(x)}{l(x)}, \]
and let $\rho(R)$ be the supremum of $\rho(x)$ as $x$ ranges over nonzero 
nonunits.  \\
a) (Carlitz \cite{Carlitz60}) We have $\rho(R) = 1$ $\iff$ $\# \operatorname{Cl} R \leq 2$. \\
b) (Valenza \cite{Valenza90}) We have $\rho(R) = \max \left( \frac{D(\operatorname{Cl} R)}{2}, 1 \right)$.
\end{thm}
\noindent
For any finite commutative group $G$, there are unique positive integers $r,n_1,\ldots,n_r$ with 
$1 < n_r \mid n_{r-1} \mid \ldots \mid n_1$ such that $G \cong \bigoplus_{i=1}^r \Z/n_i \Z$.  Put 
\[ d(G) = 1 + \sum_{i=1}^r (n_i-1). \]
Let $e_i \in \bigoplus_{i=1}^r \Z/n_i \Z$ be the element with $i$th coordinate 
$1$ and all other coordinates zero.  Then the sequence 
\[ \overbrace{e_1,\ldots,e_1}^{n_1-1},\overbrace{e_2,\ldots,e_2}^{n_2-1},
\ldots,\overbrace{e_r,\ldots,e_r}^{n_r-1} \]
shows that 
\begin{equation}
\label{DAVENPORTEQ2}
 d(G) \leq D(G). 
\end{equation}
Comparing (\ref{DAVENPORTEQ1}) and (\ref{DAVENPORTEQ2}) shows $D(G) = \# G = d(G)$ when $G$ is cyclic.   In 1969, J.E. Olson conjectured that $D(G) = d(G)$ for all $G$ and proved it in the following cases.

\begin{thm}(Olson)
\label{OLSONTHM}
For a finite commutative group $G$, $d(G) = D(G)$ holds if: \\
(i) $G$ is a direct product of \emph{two} cyclic groups; or \\
(ii) $G$ is a $p$-group (i.e., $\# G = p^a$ for some $a \in \Z^+$).
\end{thm}
\begin{proof}
Part (i) is \cite[Cor. 1.1]{Olson69b}.  Part (ii) is \cite[(1)]{Olson69a}.
\end{proof}
\noindent
However, at almost the same time Olson's conjecture was disproved.

\begin{thm}(van Emde Boas-Krusywijk \cite{EBK69})
\label{EBK}
For $G = \Z/6\Z \times \Z/3\Z \times \Z/3\Z \times \Z/3\Z$, we have $d(G) < D(G)$.  
\end{thm}
\noindent
In the intervening years there has been an explosion of work on the Davenport constant and related quantities.  Nevertheless, for most finite commutative groups $G$, the exact value of $D(G)$ remains unknown.
\\ \\
Let us turn to $g$-sum subsequences with $g \neq 0$.  There is no analogue 
of the Davenport constant here, because for for all $n \in \Z^+$, $(0,\ldots,0) \in G^n$ has length $n$ and no $g$-sum subsequence.  On the other hand, for $g \in G$ and $x \in G^n$, let 
\[N_g(x) = \# \left\{ J \subset \{1,\ldots,n\} \Big| \sum_{i \in J} x_i = g \right\}. \]

\begin{thm}
\label{NGTHM}
Let $(G,+)$ be a finite commutative group, let  $n \in \Z^+$, and let $g \in G$.
a) (\cite[Thm. 2]{Olson69b}) We have $\min_{x \in G^n} N_0(x) = \max \{1,2^{n+1-D(G)} \}$.  \\
b) (\cite[Thm. 2]{CCQWZ11}) For all $x \in G^n$, if $N_g(x) > 0$ then 
$N_g(x) \geq 2^{n+1-D(G)}$.
\end{thm}
\noindent
Now let $G = \bigoplus_{i=1}^r \Z/p^{v_i} \Z$ be a $p$-group.
\\ \\
As Schanuel observed, in this case Theorem \ref{OLSONTHM} is 
a quick consequence of Theorem \ref{SCHANTHM}.  Indeed,
suppose $n > d(G) = \sum_{i=1}^r \left(p^{v_i} -1\right)$, and represent elements of $G$ by $r$-tuples of integers $(a_1,\ldots,a_r)$.  
For $1 \leq i \leq n$ and $1 \leq j \leq r$, let \[g_j = (a_1^{(j)},\ldots,a_r^{(j)})\] and 
\[P_i(t_1,\ldots,t_n) = \sum_{j=1}^n a_i^{(j)} t_j.\]   Theorem \ref{SCHANTHM}b) applies to give 
$x \in \{0,1\}^n \setminus \{(0,\ldots,0)\}$ such that \[\sum_{j=1}^n a_i^{(j)} x_j \equiv 0 \pmod{p^{v_i}} \ \forall 1 \leq i \leq r. \] 
Then we get a zero-sum subsequence from $J = \{j \mid x_j = 1\}$.  
\\ \\
Moreover, in this case the Restricted Variable Warning's Second Theorem implies a combination of Theorem \ref{OLSONTHM} and 
Theorem \ref{NGTHM}: namely Theorem \ref{NGTHM} with $D(G)$ replaced by the 
explicit value $d(G) = \sum_{i=1}^r \left(p^{v_i} -1\right)$.   By part a), 
Theorem \ref{SRVW2THM} and Lemma \ref{FSLEMMA}, we get $N_g(x) = 0$ or 
\[ N_g(x) \geq \mm\left(2,\ldots,2;n+(n-\sum_{i=1}^r (p^{v_i}-1) \right)
= 2^{n-\sum_{i=1}^r \left(p^{v_i}-1 \right)}. \]


\subsection{Generalized Subsequences}
\textbf{} \\ \\ \noindent
The following results are the analogues of those of the previous section 
for generalized $g$-sum subsequences.  The proofs are the same.

\begin{thm}(Brink)
Let $G \cong \bigoplus_{j=1}^r \Z/p^{v_i} \Z$ be a finite commutative 
$p$-group.  Let $A_1,\ldots,A_n$ be nonempty subsets of $\Z$ such that 
each $A_i$ has pairwise incongruent elements modulo $p$. Put $A = \prod_{i=1}^n A_i$.  Assume that 
\[ \sum_{i=1}^n \left( \# A_i - 1 \right) > \sum_{j=1}^r \left( p^{v_j} - 1 \right). \]
Let $x = (x_1,\ldots,x_n) \in G^n$ be a sequence of elements in $G$. \\
a) Then $\# \{ (a_1,\ldots,a_n) \in A \mid a_1 x_1 + \ldots + a_n x_n = 0 \} \neq 1$.  \\
b) If $0 \in A$, then there is $0 \neq a = (a_1,\ldots,a_n) \in A$ such that $a_1 x_1 + \ldots + a_n x_n = 0$.  
\end{thm}

\begin{thm}
\label{OLSONESQUE}
Let $p$ be a prime, let $r,v_1,\ldots,v_r \in \Z^+$; put $G = \bigoplus_{i=1}^r \Z/p^{v_i} \Z$.  For $n \in \Z^+$, let $x = (x_1,\ldots,x_n) \in G^n$ be a sequence of elements in $G$.  Let $A_1,\ldots,A_n$ be nonempty 
subsets of $\Z$ such that for each $i$ the elements of $A_i$ are pairwise 
incongruent modulo $p$, and put $A= \prod_{i=1}^n A_i$.  For $g =(g_1,\ldots,g_r) \in G$, let 
\[N_{g,A}(x) = \# \{a = (a_1,\ldots,a_n) \in A \mid a_1 x_1 + \ldots + a_n x_n = g \}. \]
Then $N_{g,A}(x) = 0$ or 
\[ N_{g,A}(x) \geq \mm \left( \# A_1,\ldots,\# A_n; \# A_1 + \ldots + \# A_n - 
\sum_{i=1}^r (p^{v_i}-1) \right). \]
\end{thm}

\subsection{Counting Sub-(Set Systems) With Union Cardinality $0$ Modulo $q$}
\textbf{} \\ \\ \noindent
\newcommand{\FF}{\mathcal{F}}
In \cite{AKLMRS91}, AKLMRS applied 
Schanuel's Theorem to deduce a result on set systems.  This is an interesting case for these methods because (i) unlike the applications of the 
previous section the polynomials are not linear (or even obtained from 
linear polynomials by applying the Schanuel-Brink operator); (ii) there is no known purely combinatorial proof; and (iii) the bound obtained is sharp in all cases.  By applying Theorem \ref{SRVW2THM} instead of Schanuel's Theorem, we immediately derive a quantitative 
refinement of this result and also treat the ``inhomogeneous case.''
\\ \\
A \textbf{set system} is a finite sequence $\FF = (\FF_1,\ldots,\FF_n)$ 
of finite subsets of some fixed set $X$.  We say that $n$ is the \textbf{length} of $\FF$.  The \textbf{maximal degree} of $\FF$ is $\max_{x \in X} \# \{ 1 \leq i \leq n \mid x \in \FF_i\}$.  For $m$ a positive integer and $g \in \Z/m\Z$, let 
\[ N_{\FF}(m,g) = \# \{ J \subset \{1,\ldots,n\} \mid \# (\bigcup_{i \in J} \FF_i) \equiv g \pmod{m} \}, \]
and for $n,d \in \Z^+$, let
\[ \mathcal{N}_{n,d}(m) = \min N_{\FF}(m,0), \]
the minimum ranging over set systems of length $n$ and maximal degree at most $d$.  Let $f_d(m)$ be the least $n \in \Z^+$ such that for any degree $d$ set system $\FF$ of length $n$, there is a nonempty subset $J \subset \{1,\ldots,n\}$ such that $m \mid \# (\bigcup_{i \in J} \FF_i)$.   Thus 
\begin{equation}
\label{AKLMRSEQ1} 
 f_d(m) = \min \{n \in \Z^+ \mid \mathcal{N}_{n,d}(m) \geq 2 \}.
\end{equation}

\begin{lemma}(AKLMRS)
\label{AKLMRSLEMMA}
We have $f_d(m) \geq d(m-1) + 1$.
\end{lemma}
\begin{proof}
Let $A_{ij}$ be a family of pairwise disjoint sets each of cardinality $m$, as $1 \leq i \leq m-1, \ 1 \leq j \leq d$.  Let $\{v_1,\ldots,v_{m-1} \}$ be a set of cardinality $m$, disjoint from all the $A_{ij}$'s.  Then $\FF = \{ A_{ij} \cup \{v_i\} \mid 1 \leq i \leq m-1, \ 1 \leq j \leq d \}$ has length $n$ and 
for no nonempty subset $J \subset \{1,\ldots,n\}$ do we have $m \mid \#  (\bigcup_{i \in J} \FF_i)$.
\end{proof}

\begin{thm}
Let $q = p^v$ be a prime power, $g \in \Z/p^v \Z$, $d,n \in \Z^+$, and $\FF = (\FF_1,\ldots,\FF_n)$ 
a set system of maximal degree $n$.  Then: \\
a) $\mathcal{N}_{\FF}(p^v,g)$ is either $0$ or at least $2^{n-d(p^v-1)}$.  We deduce: \\
b) $\mathcal{N}_{n,d}(p^v) \geq 2^{n-d(p^v-1)}$; and thus \\
c) (AKLMRS) $f_d(q) = d(p^v-1) + 1$.
\end{thm}
\begin{proof}
a) For $\FF$ a set system of length $n$ and maximal degree at most $d$, put
\[ h(t_1,\ldots,t_n) = \sum_{\varnothing \neq J \subset \{1,\ldots,n\}} 
(-1)^{\# J + 1} \# (\bigcap_{j \in J} \FF_i) \prod_{j \in J} t_j. \]
Then $\deg h \leq d$ and $h(0) = 0$.  For any $x \in \{0,1\}^n$, let 
$J_x = \{ 1 \leq j \leq n \mid x_j = 1\}$.  The Inclusion-Exclusion Principle implies 
\[ h(x) = \# \bigcup_{j \in J_x} \FF_j, \]
so $\mathcal{N}_{\FF}(p^v,g)$ counts the number of solutions $x \in \{0,1\}^n$ to the congruence $h(t)-g \equiv 0 \pmod{p^v}$.   Applying Theorem \ref{SRVW2THM} 
establishes part a). \\
b) Taking $J = \varnothing$ shows $\mathcal{N}_{\FF}(p^v,0) \geq 1$.  Apply part a). \\
c) By part a) and (\ref{AKLMRSEQ1}), we see that $f_d(q) \leq d(q-1) + 1$.  Apply Lemma \ref{AKLMRSLEMMA}.
\end{proof}

\subsection{An EGZ-Type Theorem}
\textbf{} \\ \\ \noindent
As we saw in $\S$ 4.1, computing the Davenport constant of a finite cyclic group is an easy exercise.  A more interesting variant is to ask how large $n$ needs to be in order to ensure that any sequence of length $n$ in the group $\Z/m\Z$ has a zero-sum subsequence of length $m$.  The sequence 
\[ (\overbrace{0,\ldots,0}^{m-1},\overbrace{1,\ldots,1}^{m-1}) \]
shows that we need to take $n \geq 2m-1$.  The following converse is one of the founding results in this branch of 
additive combinatorics.

\begin{thm}(Erd\H os-Ginzburg-Ziv \cite{EGZ61})
\label{EGZ}
Every sequence of length $2m-1$ in $\Z/m\Z$ has a zero-sum subsequence of length $m$.
\end{thm}
\noindent
It is not hard to see that if Theorem \ref{EGZ} holds for positive integers $m_1$ and $m_2$ then it holds for their product, and thus one reduces to the case in which $m$ is prime.  The original work \cite{EGZ61} showed this via a combinatorial argument.  Later it was realized that one can get a quick proof using Chevalley's Theorem \cite{Bailey-Richter89}.
\\ \\
A recent paper of DAGS \cite{DAGS12} treats the analogous problem in any finite commutative $p$-group, with zero-sum subsequences replaced by generalized zero-sum subsequences in the sense of $\S$ 4.2.  As before, using Theorem \ref{SRVW2THM} we get a quantitative refinement which also includes the inhomogeneous case.
\\ \\
For a finite commutative group $G$, let $\exp G$ denote the exponent of $G$, i.e., the least common multiple of all orders of elements in $G$.

\begin{lemma}
Let $\{0\} \subset A \subset \Z$ be a finite subset, no two of whose elements are congruent modulo $p$.  There is $C_A \in \Z_{(p)}[t]$ 
of degree $\# A-1$ such that for $a \in A$, 
\[ C_A(a) =  \begin{cases} 
      0 & a = 0 \\
      1 & a \neq 0
    
   \end{cases}. \]
\end{lemma}
\begin{proof} 
We may take $C_A(t) = 1 - \prod_{a \in A \setminus \{0\}} \frac{a-t}{a}$.  
\end{proof}

\begin{thm}
\label{EGZTHM}
Let $k,r$, $v_1 \leq \ldots \leq v_r$ be positive integers, and let 
$G = \bigoplus_{i=1}^r \Z/p^{v_i} \Z$.  Let $A_1,\ldots,A_n$ be nonempty 
subsets of $\Z$, each containing $0$, such that for each $i$ the elements of $A_i$ are pairwise incongruent modulo $p$.  Put 
\[ A = \prod_{i=1}^n A_i, \ a_M = \max \# A_i. \] 
For $x \in G$, let  $\operatorname{EGZ}_{A,k}(x)$ 
be the number of $(a_1,\ldots,a_n) \in A$ such that $a_1 x_1 + \ldots + a_n x_n = x$ and $p^k \mid \# \{ 1 \leq i \leq n \mid a_i \neq 0\}$.  Then either 
$\operatorname{EGZ}_{A,k}(x) = 0$ or
\begin{equation}
\label{BIGEGZEQ}
\operatorname{EGZ}_{A,k}(x) \geq \mm(\# A_1,\ldots,\# A_n; \# A_1 + \ldots + \# A_n - \sum_{i=1}^r (p^{v_i}-1) - (a_M-1)(p^k-1)). 
\end{equation}
\end{thm}
\begin{proof}
We apply Theorem \ref{SRVW2THM} as in the proof of Theorem \ref{OLSONESQUE}.  The extra condition that the number of nonzero terms in the zero-sum generalized subsequence is a multiple of $p^k$ is enforced via the polynomial 
congruence 
\[ C_{A_1}(t_1) + \ldots + C_{A_n}(t_n) \equiv 0 \pmod{p^k}, \]
which has degree $a_M-1$.  
\end{proof}

\begin{cor}
\label{DAGS}
In Theorem \ref{EGZTHM}, let $0 \in A_1 = \ldots = A_n$, 
$k = v_r$.  Put $a = \# A_1$. \\
a) Suppose \[n \geq \exp G - 1 + \frac{D(G)}{a-1}. \] Let $R$ be such that $R \equiv -\sum_{i=1}^r (p^{v_i}-1) \pmod{a-1}$ and $0 \leq R < a-1$.  Then 
\begin{equation}
\label{DAGSINEQ}
\operatorname{EGZ}_{A,v_r}(0) \geq (R+1)a^{n+1-\exp G + \lfloor \frac{1-D(G)}{a-1} \rfloor}. 
\end{equation}
b) (\cite[Thm. 1.1]{DAGS12}) Every sequence of length $n$ in $G$ has a nonempty zero-sum generalized subsequence of length divisible by $\exp G$ when 
\begin{equation}
\label{DAGSEQ}
n \geq \exp G - 1 + \frac{D(G)}{a -1}.
\end{equation}
\end{cor}
\begin{proof}
a)  The empty subsequence ensures $\operatorname{EGZ}_{A,v_r}(0) \geq 1$, so Theorem \ref{EGZTHM} gives
\[ \operatorname{EGZ}_{A,v_r}(0) \geq \mm \left(a,\ldots,a;na - \sum_{i=1}^r \left( p^{v_i}-1 \right) - (a-1)(p^{v_r}-1) \right). \] 
We have 
\[n \geq \exp G - 1 + \frac{D(G)}{a-1} > \exp G - 1 + \frac{D(G)-1}{a-1}, \]
hence 
\[na - (D(G) -1) - (a-1)(\exp G -1) = 
na - \sum_{i=1}^r \left( p^{v_i}-1 \right) - (a-1)(p^{v_r}-1) > n. \]
By
Lemma \ref{FSLEMMA}b), we have 
\[ \mm \left(a,\ldots,a;na - \sum_{i=1}^r \left( p^{v_i}-1 \right) - (a-1)(p^{v_r}-1) \right) = (R+1)a^{n+1-\exp G + \lfloor \frac{1-D(G)}{a-1} \rfloor}. \]
b) Since $n \geq \exp G - 1 + \frac{D(G)}{a-1} > \exp G - 1 + \frac{D(G)-1}{a-1}$, we have
\[ na - \sum_{i=1}^r \left( p^{v_i}-1 \right) - (a-1)(p^{v_r}-1) > n.\]
It follows from part a) and Lemma \ref{OBVIOUSLEMMA} that 
$\operatorname{EGZ}_{A,v_r}(0) \geq 2$.  
%
\end{proof}
\noindent
In the proof of Corollary \ref{DAGS}b), rather than using part a) we could have applied Theorem \ref{BIGBRINKTHM}.  It is interesting to compare this approach with the proof of Corollary \ref{DAGS}b) given in \cite{DAGS12}.  Their argument proves the needed case of Theorem \ref{BIGBRINKTHM} by exploiting properties of binomial coefficients ${t \choose d}$ viewed as integer-valued polynomials and reduced modulo powers of $p$.  In 2006 IPM lecture notes, R. Wilson proves Theorem \ref{SCHANTHM} in this manner.  His method works to prove Theorem \ref{BIGBRINKTHM}.




\end{document}